\documentclass{amsart}
\usepackage{caption}
\usepackage{amsmath, amssymb, mathrsfs, verbatim, multirow}
\usepackage[all]{xy}
\usepackage{pifont}
\usepackage{imakeidx}
\usepackage{float}
\usepackage{color}
\usepackage{graphicx,color,epsfig}
\usepackage{pgf,tikz}
\usepackage{pgfplots}
\pgfplotsset{compat=1.15}
\usepackage{mathrsfs}
\usetikzlibrary{patterns}

\usepackage{multicol}
\usepackage[T1]{fontenc}
\usepackage{bm}
\usepackage{amsmath}
\usepackage{amsthm}
\usepackage{amssymb}
\usepackage{amscd}
\usepackage{amsfonts}
\usepackage{amsbsy}
\usepackage[normalem]{ulem}
\usepackage{graphicx,color,epsfig}
\usepackage{pgf,tikz}
\usetikzlibrary{arrows}
\usepackage{multicol}
\usepackage{subfigure}
\usepackage[shortlabels]{enumitem} 
\usepackage{tkz-tab}

\usepackage{subfigure}
\newtheorem{Teo}{Theorem}[section]
\newtheorem{Prop}[Teo]{Proposition}
\newtheorem{Lema}[Teo]{Lemma}
\newtheorem{Cor}[Teo]{Corollary}

\theoremstyle{definition}
\newtheorem{Def}[Teo]{Definition}

\newtheorem{Obs}[Teo]{Remark}

\newcommand{\Q}{\mathbb{Q}}
\newcommand{\R}{\mathbb{R}}
\newcommand{\Z}{\mathbb{Z}}
\newcommand{\N}{\mathbb{N}}

\newcommand{\lra}{\longrightarrow}
\newcommand{\Lra}{\Longrightarrow}
\newcommand{\VR}{\mathcal{O}}

\newcommand{\MI}{\mathfrak{m}}

\newcommand{\SU}{\mbox{\rm supp}}

\DeclareMathOperator{\inv}{in}
\begin{document}
\title[Valuations approaching a polynomial]{Valuations on $K[x]$ approaching a fixed irreducible polynomial}
\author{Matheus dos Santos Barnab\'e}
\author{Josnei Novacoski}
\thanks{During the realization of this project the second author was supported by a grant from Funda\c c\~ao de Amparo \`a Pesquisa do Estado de S\~ao Paulo (process number 2017/17835-9).}

\begin{abstract}
For a fixed irreducible polynomial $F$ we study the set $\mathcal V_F$ of all valuations on $K[x]$ bounded by valuations whose support is $(F)$. The first main result presents a characterization for valuations in $\mathcal V_F$ in terms of their graded rings. We also present a result which gives, for a fixed $\nu\in \mathcal V_F$ and a key polynomial $Q\in{\rm KP}(\nu)$, the maximum value that augmented valuations in $\mathcal V_F$ can assume on $Q$. This value is presented explicitly in terms of the slopes of the Newton polygon of $F$ with respect to $Q$. Finally, we present some results about Artin-Schreier extensions that illustrate the applications that we have in mind for the results in this paper.
\end{abstract}

\keywords{Key polynomials, graded rings, Newton polygons}
\subjclass[2010]{Primary 13A18}

\maketitle
\section{Introduction}
Let $(K,\nu_0)$ be a valued field and $\Gamma_0$ its group of values. Fix a totally ordered divisible group $\Gamma$ containing $\Gamma_0$ and $R=K[x]$.   Let
$$
\mathcal V=\{\nu_0\}\cup\{\nu:R\lra \Gamma_\infty\mid\nu\text{ is a valuation extending }\nu_0\}.
$$
For the corresponding definitions see Section \ref{prelimi}. Consider the partial order on $\mathcal V$ given by $\nu_0\leq \nu$ for every $\nu\in \mathcal V$ and for $\nu,\mu\in\mathcal V\setminus\{\nu_0\}$, we set $\nu\leq \mu$ if and only if $\nu(f)\leq\mu(f)$ for every $f\in K[x]$.

The problem of understanding the structure of $\mathcal V$ has been extensively studied and many objects introduced in order to describe it. For instance, \textit{key polynomials} (as in \cite{Mac_1}, \cite{Mac_2}, \cite{SopivNova}, \cite{Vaq_1} and \cite{Vaq_2}, among others) are very useful objects. From this theory, it follows that for every $\nu\in\mathcal V$, there exists a well-ordered set  $\{\nu_i\}_{i\in I}\subseteq \mathcal V$, with $\nu_0$ its smallest element and each $\nu_i$ constructed iteratively, such that for every $f\in K[x]$ there exists $i\in I$ for which
\[
\nu(f)=\nu_j(f)\mbox{ for every }j\in I\mbox{ with }i\leq j.
\]
This means that $\nu$ can be seen as the ``limit of the sequence $\{\nu_i\}_{i\in I}$''. Moreover, each of these valuations $\nu_i$ has special properties, for instance they are \textit{valuation-transcendental} (see Theorem 2.3 of \cite{Nart1} and Theorem 1.1 of \cite{Caio}).

Another type of objects used to understand $\mathcal V$ are \textit{pseudo-convergent sequences} as introduced by Kaplansky in \cite{Kapl}. It follows from Theorems 2 and 3 of \cite{Kapl} that if $K$ is algebraically closed, then for every $\nu\in\mathcal V$, either $\nu$ is \textit{valuation-transcendental} or there exists a \textit{pseudo-convergent sequence} $\{a_i\}_{i\in I}$ in $K$ such that for every $f\in K[x]$ there exists $i\in I$ such that $\nu(f)=\nu_0(f(a_j))$ for every $j\in I$ with $j\geq i$. More recently, Peruginelli and Spirito show (see \cite{Giul_2} and \cite{Giul_1}) that if we allow a more general type of sequences (called \textit{pseudo-monotone}), then it is enough that the completion of $K$, with respect to $\nu_0$, is algebraically closed. Also, the concept of minimal pairs (as defined in \cite{Kand}) can be used in a similar way. The comparison between these different objects can be found in \cite{NovaMP} and \cite{Nov1}.

It is important to understand the maximal elements of $\mathcal V$. For instance, it follows from Theorem 4.4 of \cite{Nart1} that if the \textit{graded ring} of $\nu$ is algebraic over the graded ring of $\nu_0$, then $\nu$ is maximal in $\mathcal V$. Also, Theorem 1.2 of \cite{Caio} tells us that this happens if and only if $\nu$ is strictly greater than its \textit{truncation} on any $q\in K[x]$. Another type of elements in $\mathcal V$ that are maximal are the \textit{non-Krull valuations}. In this case, we have
\[
\SU(\nu):=\{g\in K[x]\mid \nu(g)=\infty\}=(F)
\]
for some irreducible polynomial $F\in K[x]$.

For a fixed monic and irreducible polynomial $F\in K[x]$ we set
\[
\mathcal{V}_{F_\infty}=\{\mu\in \mathcal V\mid \SU(\mu)=(F)\}
\]
and
\[
\mathcal V_F:=\{\nu\in \mathcal V\mid \nu\leq \mu \mbox{ for some }\mu\in \mathcal{V}_{F_\infty}\}.
\]
A valuation in $\mathcal V_F$ is said to \textbf{approach} the polynomial $F$. The subset $\mathcal{V}_{F_\infty}$ consists of all the maximal elements in $\mathcal V_F$. The elements in $\mathcal{V}_{F_\infty}$ correspond to extensions of $\nu_0$ to the simple extension $K(\eta)=K[x]/(F)$ of $K$ and those correspond to the irreducible factors of $F$ in the \textit{henselization} of $(K,\nu_0)$.

Our first goal is to characterize the elements of $\mathcal V_F$. For any valuation $\nu$ on a ring $R$ we can construct the graded ring ${\rm gr}_\nu(R)$ associated to it and a natural map $R\setminus\SU(\nu)\lra \mathcal {\rm gr}_\nu(R)$ denoted by $f\mapsto \inv_\nu(f)$ (see Section \ref{prelimi}). For $\nu\in\mathcal V\setminus\{\nu_0\}$ we will denote by $\mathcal G_\nu$ the graded ring ${\rm gr}_\nu(K[x])$ and by $\mathcal U_\nu$ the set of all units of $\mathcal G_\nu$.
\begin{Teo}\label{Propochata}
A non-maximal valuation $\nu\in \mathcal V$ belongs to $\mathcal V_F$ if and only if $\inv_\nu(F)\notin \mathcal U_\nu$. 
\end{Teo}
Take $\nu\in \mathcal V\setminus\{\nu_0\}$ and a monic polynomial $Q\in K[x]$. Then $Q$ is said to be a \textbf{Mac Lane-Vaqui\'e key polynomial} for $\nu$ if $\inv_\nu(Q)$ is a prime element of $\mathcal G_\nu$ and for every $f\in K[x]$ we have
\[
\inv_\nu(Q)\mid \inv_\nu(f)\Lra \deg(Q)\leq \deg(f).
\] 
We denote by KP($\nu$) the set of all Mac Lane-Vaqui\'e key polynomials for $\nu$.

For every homogeneous irreducible element $\phi \in\mathcal G_\nu$ there exists $Q\in {\rm KP}(\nu)$ such that $\phi=u\cdot \inv_\nu(Q)$ for some $u\in \mathcal U_\nu$ (see Theorem 6.8 of \cite{Nart1}). 

\begin{Prop}\label{proposparrame}
Take a non-maximal valuation $\nu\in \mathcal V_F$. For each homogeneous irreducible factor $\phi$ of $\inv_\nu(F)$ take $Q\in {\rm KP}(\nu)$ such that $\phi=u\cdot \inv_\nu(Q)$ for some $u\in \mathcal U_\nu$. Then there exists $\nu'\in \mathcal V$ such that
\[
\nu<\nu',\ \nu(Q)<\nu'(Q), \ \nu(F)<\nu'(F)\mbox{ and }\inv_{\nu'}(Q)\mid \inv_{\nu'}(F).
\]
\end{Prop}

An important consequence of Theorem \ref{Propochata} and the above proposition is the following.

\begin{Teo}\label{corusadoprafrente}
Take a non-maximal valuation $\nu\in \mathcal V_F$ and write
\[
\inv_\nu(F)=u\cdot\prod_{i=1}^n\phi_i^{r_i} 
\]
where each $\phi_i$ is homogeneous and irreducible, $u\in\mathcal U_\nu$ and $\phi_i\nmid \phi_j$ if $i\neq j$. Then there exist distinct valuations $\mu_1,\ldots,\mu_n\in \mathcal V_{F_\infty}$ such that $\nu< \mu_i$ for every $i$, $1\leq i\leq n$. 
\end{Teo}

An immediate consequence of the above result is the following.

\begin{Cor}
If $\mu$ is the only extension of $\nu_0$ to $L=K[x]/(F)$, then for every $\nu\in\mathcal V$, with $\nu< \mu$, have $\inv_\nu(F)=u\cdot\phi^r$ for some irreducible element $\phi$ in $\mathcal G_\nu$ and $u\in \mathcal U_\nu$. 
\end{Cor}  

The next theorem is the second main result of this paper.
\begin{Teo}\label{mainthem}
Take a non-maximal valuation $\nu\in \mathcal V_F$. For each homogeneous irreducible factor $\phi\in\mathcal G_\nu$ of $\inv_\nu(F)$ write $\phi=u\cdot\inv_\nu(Q)$ with $u\in \mathcal U_\nu$ and $Q\in {\rm KP}(\nu)$. Then, there exist $\alpha_1\in\Gamma$ and $\mu\in \mathcal V_F$ such that $\nu(Q)<\alpha_1=\mu(Q)$ and for every $\mu'\in \mathcal V_F$ we have
\[
\mu\leq \mu'\Lra \mu'(Q)=\alpha_1.
\]
\end{Teo}

An important feature in the above theorem is that we can present explicitly the value $\alpha_1$. Namely, this value is minus the slope of the first side of the \textit{Newton polygon of $F$ with respect to $Q$}. Hence, this value can depend on the choice of $Q$.

The main purpose of this paper is to understand properties of the extensions of $\nu_0$ in terms of key polynomials, Newton polygons and graded rings. The results presented above are for a fixed irreducible polynomial $F$. If we run over all such polynomials, then we cover all the extensions of $\nu_0$ to simple algebraic extensions of $K$. These results present a criterion to decide whether a given valuation is bounded by a valuation on a given simple algebraic extension. Moreover, we present an explicit way to calculate the maximum value for which we can augment such valuations. We intend to use this to present an algorithm to fully classify all the extensions of a given valuation to algebraic extensions. This is work in progress.

One property of extensions that we want to classify is the defect in terms of the degrees of the limit key polynomials that define it. Such relation is well-known in the case of \textit{unique extension}. We hope to use our results to understand the defect of an extension when it is not the unique one.

In this direction, we present an application (to Artin-Schreier extensions) that illustrates our method. Namely, we prove the following.
\begin{Teo}\label{kuhlmannrequs}
Let $(K,\nu_0)$ be a valued field with ${ \rm char}(K)=p>0$, with value group $\Gamma_0$. Let $F(x)=x^p-x-a\in K[x]$ be an irreducible Artin-Schreier polynomial and $L=K[x]/(F)$. Assume that $\nu_0$ is a rank one valuation. We have the following.
\begin{description}
\item[(i)] There exists $b\in K$ such that $\nu_0(F(b))>0$ if and only if there exist $p$ many distinct extensions of $\nu_0$ to $L$.
\item[(ii)] Assume that $\nu_0(F(b))\leq 0$ for every $b\in K$. Set
\[
S:=\left\{\frac{\nu_0(F(b))}{p}\mid b\in K\right\}\subseteq \frac{1}{p}\Gamma_0.
\]
Fix an extension $\nu$ of $\nu_0$ to $L$ and let $e$ and $f$ be the ramification index and inertia degree of $(L|K,\nu)$, respectively. Then we have the following.
\begin{itemize}
\item If $S\not\subseteq \Gamma_0$, then $e=p$.
\item If $S\subseteq \Gamma_0$ and $S$ has a maximum, then $f=p$.
\item If $S\subseteq \Gamma_0$ and $S$ does not have a maximum, then $\nu$ is the unique extension of $\nu_0$ to $K$ and $e=f=1$. In particular,
\[
d:=\frac{[L:K]}{ef}=p.
\]
\end{itemize}
\end{description}
\end{Teo}

We emphasize that some of the results in this paper are not new. For instance, in Th\'eor\`eme 0.1 of \cite{Va_3}, Vaqui\'e presents a characterization of elements in $\mathcal V_F$ similar to the one in Theorem \ref{Propochata}. However, our characterization and its proof are simpler. Also, the idea of using Newton polygons to understand key polynomials and \textit{generating sequences} is not new. It appears in various works, for instance \cite{CMT}, \cite{NMONa} and \cite{GOS}. Moreover, part of Theorem \ref{kuhlmannrequs} appeared already in \cite{Kuhl}. 

This paper is divided as follows. In Section \ref{prelimi} we introduce the basic definitions and notations to be used in the sequel. In Section \ref{Newtonpoly} we present the concept of Newton polygons. This concept gives a geometric interpretation for the main results. Although some of the proofs do not use explicitly Newton polygons, they provide a geometric and visual interpretation for each step.

Finally, in Section \ref{proofofthem} we present the proofs of the main results. In the last section we present the proof of Theorem \ref{kuhlmannrequs}. 

\textbf{Acknowledgements.} We would like to thank Franz-Viktor Kuhlmann for a careful reading, for providing useful suggestions and for pointing out a few mistakes in an earlier version of this paper. We also thank the anonymous referee for corrections and suggestions that made this paper more accurate and readable.

\section{Preliminaries}\label{prelimi}
Throughout this paper, $\N$ will denote the set of positive integers, $\N_0$ the set of non-negative integers and $\Q$ the set of rational numbers.

\begin{Def}
Take a commutative ring $R$ with unity. A \index{Valuation}\textbf{valuation} on $R$ is a mapping $\nu:R\lra \Lambda_\infty :=\Lambda \cup\{\infty\}$ where $\Lambda$ is an ordered abelian group (and the extension of addition and order to $\infty$ in the obvious way), with the following properties:
\begin{description}
	\item[(V1)] $\nu(ab)=\nu(a)+\nu(b)$ for all $a,b\in R$.
	\item[(V2)] $\nu(a+b)\geq \min\{\nu(a),\nu(b)\}$ for all $a,b\in R$.
	\item[(V3)] $\nu(1)=0$ and $\nu(0)=\infty$.
\end{description}
\end{Def}

Let $\nu: R \lra\Lambda_\infty$ be a valuation. The set $\SU(\nu)=\{a\in R\mid \nu(a )=\infty\}$ is called the \textbf{support of $\nu$}.  A valuation $\nu$ is a \index{Valuation!Krull}\textbf{Krull valuation} if $\SU(\nu)=\{0\}$. The \textbf{value group of $\nu$} is the subgroup of $\Lambda$ generated by 
$\{\nu(a)\mid a \in R\setminus \SU(\nu) \}
$ and is denoted by $\nu R$. If $R$ is a field, then we define the \textbf{valuation ring} of $\nu$  by $\VR_\nu:=\{ a\in R\mid \nu(a)\geq 0 \}$. The ring $\VR_\nu$ is a local ring with unique maximal ideal $\MI_\nu:=\{a\in R\mid \nu(a)>0 \}. $ We define the \textbf{residue field} of $\nu$ to be the field $\VR_\nu/\MI_\nu$ and denote it by $R\nu$. The image of $a\in \VR_\nu$ in $R\nu$ is denoted by $a\nu$. If $R=K[x]$, we will identify a non-Krull valuation $\nu$ with the induced valuation on $L=K[x]/(F)$ where $\SU(\nu)=(F)$.

\begin{Prop}[Corollary 2.5 (2) of \cite{MLV}]\label{propsobucara}
Assume $\eta,\nu,\mu\in\mathcal V$ are such that $\eta<\nu<\mu$. For $f\in K[x]$, if $\eta(f)=\nu(f)$, then $\nu(f)=\mu(f)$.
\end{Prop}

\begin{Cor}\label{Corpahultim}
Let $\{\nu_i\}_{i\in I}$ be a totally ordered set in $\mathcal V$. For every $f\in K[x]$ we have that either $\{\nu_i(f)\}_{i\in I}$ is strictly increasing, or there exists $i_0\in I$ such that $\nu_i(f)=\nu_{i_0}(f)$ for every $i\in I$ with $i\geq i_0$.
\end{Cor}
\begin{proof}
Suppose that there exists $j,i_0\in I$ with $j<i_0$ such that $\nu_j(f)=\nu_{i_0}(f)$. For any $i>i_0$ we have that $\nu_j<\nu_{i_0}<\nu_i$. Hence, by Proposition \ref{propsobucara} we conclude that $\nu_i(f)=\nu_{i_0}(f)$.
\end{proof}

Let $\mathfrak v=\{\nu_i\}_{i\in I}$ be a totally ordered set in $\mathcal V$, without last element. For $f\in K[x]$ we say that $f$ is \textbf{$\mathfrak v$-stable} if there exists $i_f\in I$ such that
\[
\nu_i(f)=\nu_{i_f}(f)\mbox{ for every }i\in I\mbox{ with }i\geq i_f.
\]
Consider the set
\[
\mathcal C(\mathfrak v):=\{f\in K[x]\mid f\mbox{ is }\mathfrak v\mbox{-stable}\}.
\]
For every $f\in \mathcal C(\mathfrak v)$ we set $\mathfrak v(f)=\nu_{i_f}(f)$. For $f,q\in K[x]$, $q$ non-constant, there exist uniquely determined $f_0,\ldots,f_r\in K[x]$ with $\deg(f_i)<\deg(q)$ for every $i$, $0\leq i\leq r$, such that
\[
f=f_0+\ldots+f_rq^r.
\]
This expression is called the \textbf{$q$-expansion of $f$}.
\begin{Teo}\label{themgenraliaugme}
Suppose that there exists a polynomial not $\mathfrak v$-stable and take $Q$ monic and with smallest degree with this property. Take $\gamma\in \Gamma_\infty$ such that $\gamma>\nu_i(Q)$ for every $i\in I$ and consider the map
\[
\mu(f)=\min_{0\leq j\leq r}\{\mathfrak v(f_j)+j\gamma\},
\]
where $f=f_0+\ldots+f_rQ^r$ is the $Q$-expansion of $f$. Then $\mu$ is a valuation.
\end{Teo}
\begin{proof}
For $f,g\in K[x]$ with $\deg(f),\deg(g)<\deg(Q)$ write $fg=aQ+c$ with $\deg(c)<\deg(Q)$. We claim that $\mathfrak v(fg)=\mathfrak v(c)$. If this were not the case, we would have that
\[
\nu_i(aQ)=\min\{\nu_i(c),\nu_i(fg)\}=\min\{\mathfrak v(c),\mathfrak v(fg)\}\mbox{ for }i\mbox{ large enough}.
\]
This is a contradiction to the fact that $\{\nu_i(Q)\}_{i\in I}$ is increasing. Since $\mathfrak v(fg)=\nu_i(fg)$, $\mathfrak v(c)=\nu_i(c)$ and $\mathfrak v(a)=\nu_i(a)$ for $i$ large enough and $\gamma>\nu_i(Q)$ for every $i\in I$ we have that
\[
\mathfrak v(a)+\gamma>\nu_i(a)+\nu_i(Q)\geq\min\{\nu_i(c),\nu_i(fg)\}=\mathfrak v(fg)=\mathfrak v(c).
\]
By Theorem 1.1 of \cite{Nov1}, we conclude that $\mu$ is a valuation.
\end{proof}

\begin{Def}
In the case of the above theorem, we will denote $\mu$ by
\[
\mu=[\{\nu_i\}_{i\in I};\mu(Q)=\gamma].
\]
\end{Def}

Let $R$ be an integral domain and $\nu$ a valuation on $R$. For each $\gamma\in \nu(R)$, $\gamma<\infty$, we consider the abelian groups
\begin{equation*}
\mathcal{P}_{\gamma}=\{a \in R \mid \nu(a) \geq \gamma\} \mbox{ and } \mathcal{P}^{+}_{\gamma}=\{a \in R \mid \nu(a) > \gamma\}.
\end{equation*}

\begin{Def}
The \textbf{graded ring of $R$ associated to $\nu$} is defined as
\begin{equation*}
{\rm gr}_\nu(R)= \displaystyle  \bigoplus_{\gamma \in \nu(R)} \mathcal{P}_{\gamma}/\mathcal{P}^{+}_{\gamma}.
\end{equation*}
\end{Def}
It is not difficult to show that ${\rm gr}_\nu(R)$ is an integral domain. For $a\in R\setminus\{\SU(\nu)\}$ we will denote by ${\rm in}_\nu(a)$ the image of $a$ in
\[
\mathcal P_{\nu(a)}/\mathcal P_{\nu(a)}^+\subseteq{\rm gr}_\nu(R).
\]
If $a\in\SU(\nu)$ we set $\inv_\nu(a)=0$.

\begin{Teo}[Theorem 4.2 of \cite{Mac_1}]\label{lemasobreexenso}
Take $\nu\in\mathcal V\setminus\{\nu_0\}$, $Q\in{\rm KP}(\nu)$ and $\alpha\in\Gamma_\infty$ with $\alpha>\nu(Q)$. Consider the map
\[
\mu\left(f\right)=\min_{0\leq i\leq r}\{\nu(f_i)+i\alpha\},
\]
where $f=f_0+\ldots+f_rQ^r$ is the $Q$-expansion of $f$. Then $\mu$ is a valuation on $K[x]$. Moreover, $\nu<\mu$ and for $f\in K[x]$ we have $\nu(f)<\mu(f)$ if and only if ${\rm in}_\nu(Q)\mid {\rm in}_\nu(f)$.
\end{Teo}
The valuation constructed above will be called an \textbf{augmentation of $\nu$ with respect to $Q$} and denoted by
\[
\mu=[\nu;\mu(Q)=\alpha].
\]
\begin{Obs}
In \cite{Mac_1} the above theorem is stated and proved for rank one valuations. However, the same proof works for any valuation. 
\end{Obs}

\begin{Obs}\label{obsmaclane}
Let $\nu\in \mathcal V$ and suppose that there exists $\mu\in\mathcal V$ such that $\nu<\mu$. If $Q$ is a monic polynomial of smallest degree such that $\nu(Q)<\mu(Q)$, then it follows from the proof of Theorem 8.1 of \cite{Mac_2} that $Q\in {\rm KP}(\nu)$.
\end{Obs}

\section{Newton polygons}\label{Newtonpoly}
Let $f,q\in K[x]$, $q$-monic and non-constant, and $\nu\in\mathcal V\setminus\{\nu_0\}$. Let $f=f_0+\ldots+f_rq^r$ be the $q$-expansion of $f$. Set
\[
\nu_q(f)=\min_{0\leq i\leq r}\{\nu(f_iq^i)\}
\]
and
\[
S_{q,\nu}(f)=\{i\mid 0\leq i\leq r\mbox{ such that }\nu_q(f)=\nu(f_iq^i)\}.
\]
Denote by $\Gamma_\Q$ the divisible hull $\Gamma\otimes_\Z\Q$ of $\Gamma$. Consider the set
\[
V_{f,q}:=\{(i,\nu(f_i))\mid 0\leq i\leq r\mbox{ and }\nu(f_i)\neq \infty\}\subseteq \Q\times \Gamma_\Q.
\]
Assume that $f_0,f_r\notin\SU(\nu)$. 
\begin{Def}
The Newton polygon $\Delta_{q,\nu}(f)$ of $f$ with respect to $q$ and $\nu$ in $\Q\times \Gamma_\Q$ is defined to be the highest convex polygonal line from $(0,\nu(f_0))$ to $(r,\nu(f_r))$ that passes through or below all of the points in $V_{f,q}$.
\end{Def}
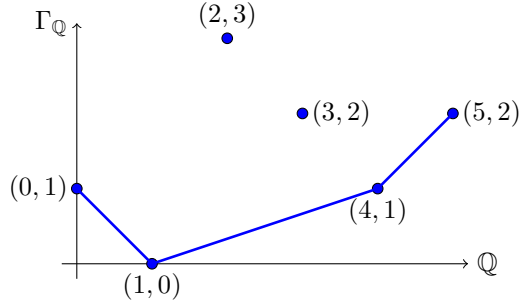
\begin{figure}[H]
    \centering
\begin{tikzpicture}
    \draw[->] (-0.2,0) -- (5.2,0) node[right] {$\Q$};
    \draw[->] (0,-0.2) -- (0,3.2) node[left] {$\Gamma_\Q$};
  \filldraw[fill=blue] (0,1) circle (2pt) node[anchor=east] {$(0,1)$};
  \filldraw[fill=blue] (1,0) circle (2pt) node[below] {$(1,0)$};
  \filldraw[fill=blue] (2,3) circle (2pt) node[above] {$(2,3)$};
  \filldraw[fill=blue] (3,2) circle (2pt) node[right] {$(3,2)$};
  \filldraw[fill=blue] (4,1) circle (2pt) node[below] {$(4,1)$};
  \filldraw[fill=blue] (5,2) circle (2pt) node[right] {$(5,2)$};
  \draw [line width=1.pt,color=blue] (0,1)-- (1,0);
  \draw [line width=1.pt,color=blue] (1,0)-- (4,1);
  \draw [line width=1.pt,color=blue] (4,1)-- (5,2);
  
\end{tikzpicture}
\caption{Newton polygon $\Delta_{q,\nu}(f)$ for $f=f_0+f_1q+\ldots+f_5q^5$ with $
\nu(f_0)=1$, $\nu(f_1)=0$, $\nu(f_2)=3$, $\nu(f_3)=2$, $\nu(f_4)=1$ and $\nu(f_5)=2$}
\end{figure}
\begin{Obs}
When there is no risk of confusion we will denote $\Delta_{q,\nu}(f)$ by $\Delta_q(f)$ and $S_{q,\nu}$ by $S_q(f)$.
\end{Obs}
For $i,k$, $0\leq i<k\leq r$, with $\nu(f_i)\neq \infty$ and $\nu(f_k)\neq \infty$, we denote by $I_{i,k}$ the segment in $\Q\times \Gamma_\Q$ joining $(i,\nu(f_i))$ and $(k,\nu(f_k))$.
\begin{Obs}
The Newton polygon of $f$ with respect to $q$ is the union of $s$ many sets $I_1,\ldots,I_s$, with the following properties.
\begin{description}
\item[(i)] $s\leq r$.
\item[(ii)] There exist $i_0,\ldots,i_s$, $0=i_0<i_1<\ldots<i_s=r$, such that
\[
I_j:=I_{i_{j-1},i_j}\mbox{ for }1\leq j\leq s.
\]
\end{description}
The sets $I_1,\ldots,I_s$ are called the \textbf{sides of $\Delta_q(f)$}.
\end{Obs}
\begin{Def}
For each $j$, $1\leq j\leq s$, we define the \textbf{slope of the side $I_j$} by
\[
\alpha_j:=\frac{\nu(f_{i_j})-\nu(f_{i_{j-1}})}{i_j-i_{j-1}}\in \Gamma_\Q.
\]
We also say that the length of the side $I_j$ is $i_j-i_{j-1}$.
\end{Def}

\begin{Obs}
For each $j$, $1\leq j\leq s$, and $i$, $0\leq i\leq r$, we have the following.
\begin{description}
\item[(i)]
\[
\nu(f_i)-i\alpha_j\geq \nu(f_{i_j})-i_j\alpha_j=\nu(f_{i_{j-1}})-i_{j-1}\alpha_j.
\]
\item[(ii)] If $0\leq i<i_{j-1}$ or $i_j<i\leq r$, then
\[
\nu(f_i)-i\alpha_j> \nu(f_{i_j})-i_j\alpha_j.
\]
\end{description}
\end{Obs}

\begin{Def}
For each $\alpha\in\Gamma_\Q$ and $P=(a,b)\in\Q\times \Gamma_{\Q}$ we denote by $L_{\alpha,P}$ the equation of the line passing through $P$ with slope $\alpha$, i.e., $L_{\alpha,P}$ is the map
\[
L_{\alpha,P}:\Q\lra \Gamma_{\Q}
\]
given by
\[
L_{\alpha,P}(x)=\alpha x+(b-\alpha a).
\]
\end{Def}

\begin{Lema}\label{Lemaquefazfuncioan}
For every $i$, $1\leq i\leq r$, set $P_i=(i,\nu(f_i))$. If $\nu(q)=-\alpha$, then
\[
\nu(f_iq^i)=L_{\alpha,P_i}(0).
\]
\end{Lema}
\begin{proof}
We have
\[
L_{\alpha,P_i}(0)=\nu(f_i)-\alpha i=\nu(f_i)+i\nu(q)=\nu(f_iq^i).
\]
\end{proof}

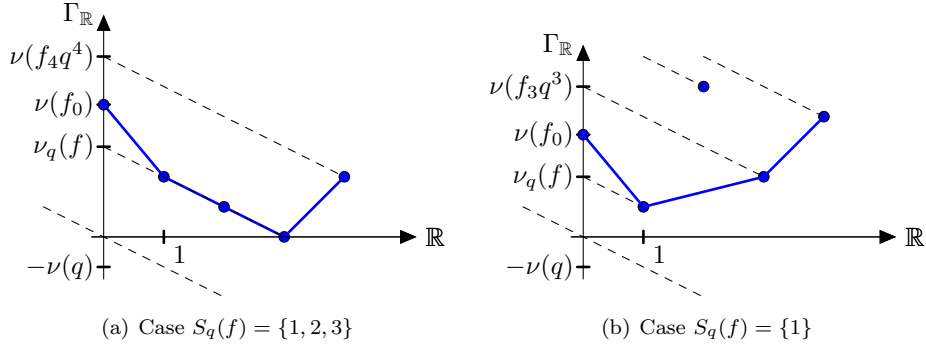
\begin{figure}[!htb]
\centering
\subfigure[Case $S_q(f)=\{1,2,3\}$]{
\begin{tikzpicture}[line cap=round,line join=round,>=triangle 45,x=0.8cm,y=0.8cm]
    \draw[->] (-0.2,0) -- (5.2,0) node[right] {$\R$};
    \draw[->] (0,-0.7) -- (0,3.7) node[left] {$\Gamma_\R$};
  
  \filldraw[fill=blue] (1,1) circle (2pt);
  \filldraw[fill=blue] (2,0.5) circle (2pt);
  \filldraw[fill=blue] (3,0) circle (2pt);
  \filldraw[fill=blue] (4,1) circle (2pt);

  \draw [line width=1.pt,color=blue] (0,2.2)-- (1,1);
  \draw [line width=1.pt,color=blue] (1,1)-- (3,0);
  \draw [line width=1.pt,color=blue] (3,0)-- (4,1);

\draw[->] (-0.2,0) -- (5.2,0) node[right] {$\R$};
   \draw [dashed] (0,1.5)-- (3,0);
   \draw [dashed] (0,3)-- (4,1);      

   \draw [line width=1.pt,color=black] (-0.1,1.5)-- (0.1,1.5);
   \draw [line width=1.pt,color=black] (-0.1,3)-- (0.1,3);
   \draw [line width=1.pt,color=black] (-0.1,2.2)-- (0.1,2.2);

  \filldraw[fill=blue] (0,2.2) circle (2pt) node[left] {$\nu(f_0)$};
  \filldraw[fill=blue] (0,3) circle (0pt) node[left] {$\nu(f_4q^4)$};
  \filldraw[fill=blue] (0,1.5) circle (0pt) node[left] {$\nu_q(f)$};  
        
  \filldraw[fill=black] (0,-0.5) circle (0pt) node[left] {$-\nu(q)$};
  \filldraw[fill=black] (1,0) circle (0pt) node[below right] {$1$};
  \draw [line width=1.pt,color=black] (-0.1,-0.5)-- (0.1,-0.5);        
  \draw [dashed] (-1,0.5)-- (2,-1);      
  \draw [line width=1.pt,color=black] (1,-0.1)-- (1,0.1);

\end{tikzpicture}

}
\subfigure[Case $S_q(f)=\{1\}$]{
\begin{tikzpicture}[line cap=round,line join=round,>=triangle 45,x=0.8cm,y=0.8cm]
    \draw[->] (-0.2,0) -- (5.2,0) node[right] {$\R$};
    \draw[->] (0,-0.7) -- (0,3.2) node[left] {$\Gamma_\R$};
  
  \filldraw[fill=blue] (1,0.5) circle (2pt);
  \filldraw[fill=blue] (2,2.5) circle (2pt);
  \filldraw[fill=blue] (3,1) circle (2pt);
  \filldraw[fill=blue] (4,2) circle (2pt);

  \draw [line width=1.pt,color=blue] (0,1.7)-- (1,0.5);
  \draw [line width=1.pt,color=blue] (1,0.5)-- (3,1);
  \draw [line width=1.pt,color=blue] (3,1)-- (4,2);

\draw[->] (-0.2,0) -- (5.2,0) node[right] {$\R$};
  \draw [dashed] (0,1)-- (1,0.5);
   \draw [dashed] (1,3)-- (2,2.5);
   \draw [dashed] (0,2.5)-- (3,1);
   \draw [dashed] (2,3)-- (4,2);      

   \draw [line width=1.pt,color=black] (-0.1,1.7)-- (0.1,1.7);
   \draw [line width=1.pt,color=black] (-0.1,1)-- (0.1,1);
   \draw [line width=1.pt,color=black] (-0.1,2.5)-- (0.1,2.5);   
  \filldraw[fill=blue] (0,1.7) circle (2pt) node[left] {$\nu(f_0)$};
  \filldraw[fill=blue] (0,1) circle (0pt) node[left] {$\nu_q(f)$};  
  \filldraw[fill=blue] (0,2.5) circle (0pt) node[left] {$\nu(f_3q^3)$};

  \filldraw[fill=black] (0,-0.5) circle (0pt) node[left] {$-\nu(q)$};
  \filldraw[fill=black] (1,0) circle (0pt) node[below right] {$1$};
  \draw [line width=1.pt,color=black] (-0.1,-0.5)-- (0.1,-0.5);        
  \draw [dashed] (-1,0.5)-- (2,-1);      
  \draw [line width=1.pt,color=black] (1,-0.1)-- (1,0.1);
\end{tikzpicture}
}
\caption{Examples when $S_{q}(f)$ is a singleton or not}\label{ilisytraquero}
\end{figure}

\begin{Obs}
Figure \ref{ilisytraquero} illustrates the next result.
\end{Obs}
\begin{Lema}\label{singl}
The set $S_q(f)$ is not a singleton if and only if $\alpha:=-\nu(q)$ is the slope of a side of $\Delta_q(f)$.
\end{Lema}
\begin{proof}
Assume that $S_q(f)$ is not a singleton and take $i,j\in S_q(f)$ with $j\neq i$. This means that
\[
L_{\alpha,P_i}(0)=\nu(f_iq^i)=\nu(f_jq^j)=L_{\alpha,P_j}(0).
\]
Hence, the lines passing through $(i,\nu(f_i))$ e $(j,\nu(f_j))$ with slope $\alpha$ coincide. Since $i,j\in S_q(f)$ this implies that $-\nu(q)$ is the slope of a side of $\Delta_q(f)$.

If $\alpha$ is the slope of a side of $\Delta_q(f)$, then for some $j$, $1\leq j\leq s$, we have that $\alpha$ is the slope of $I_j$. Moreover, by definition of $\Delta_q(f)$ we have that
\[
\nu\left(f_{i_j}q^{i_j}\right)=\nu\left(f_{i_{j-1}}q^{i_{j-1}}\right)\leq \nu\left(f_iq^i\right)\mbox{ for every }i, 0\leq i\leq r.
\]
Hence, $S_q(f)$ is not a singleton.
\end{proof}

\begin{Prop}\label{Popdiscumath}
If $\nu_q(f)<\nu(f)$, then $-\nu(q)$ is equal to the slope of a side of $\Delta_q(f)$.
\end{Prop}
\begin{proof}
If $S_q(f)$ is a singleton, then $\nu_q(f)=\nu(f)$. By Lemma \ref{singl} we conclude that $-\nu(q)$ is equal to the slope of a side of $\Delta_q(f)$.
\end{proof}

\begin{Prop}\label{propfrelagradeainvandponl}
Take $\nu\in\mathcal V$ and assume that ${\rm KP}(\nu)\neq \emptyset$. Take $Q\in {\rm KP}(\nu)$ of smallest degree. For any $f\in K[x]$ denote by $\alpha_f$ the slope of the first side of $\Delta_{Q,\nu}(f)$. Then $\inv_{\nu}(f)\in \mathcal U_\nu$ if and only if $\nu(Q)>-\alpha_f$. 
\end{Prop}
\begin{proof}
Let $f=f_0+\ldots+f_rQ^r$ be the $Q$-expansion of $f$. By Lemma \ref{Lemaquefazfuncioan} and Proposition 2.3 of \cite{Nart1}, we have that $\nu(Q)>-\alpha_f$ if and only if $\nu(f-f_0)>\nu(f)$. The result follows from Proposition 3.5 of \cite{Nart1}.
\end{proof}
\section{Proofs of the main results}\label{proofofthem}
We will need the following lemma.
\begin{Lema}\label{lemausadoduasvezes}
Assume that $\mathfrak{v}=\{\nu_i\}_{i\in I}\subseteq \mathcal V$ is a totally ordered set, without last element, such that $\{\nu_i(F)\}_{i\in I}$ is strictly increasing. Then there exists $\nu\in\mathcal V$ such that the following are satisfied.
\begin{description}
\item[(i)] $\nu_i<\nu$ and $\nu_i(F)<\nu(F)$ for every $i\in I$.
\item[(ii)] Either $F\in \SU(\nu)$ or $\inv_\nu(F)\notin \mathcal U_\nu$.
\end{description}
\end{Lema}
\begin{proof}
If $g\in K[x]$ is $\mathfrak v$-stable, then we set $\nu(g):=\mathfrak v(g)$. By Corollary \ref{Corpahultim}, for every $g\in K[x]$ we have that either $\{\nu_i(g)\}_{i\in I}$ is increasing or $g$ is $\mathfrak v$-stable.

If every polynomial of degree smaller than $\deg(F)$ is $\mathfrak{v}$-stable, then we take $\alpha\in \Gamma_\infty$ with $\alpha>\nu_i(F)$ for every $i\in I$, and set
\[
\nu=[\{\nu_i\}_{i\in I};\nu(F)=\alpha].
\]
It is obvious that $\nu$ satisfies \textbf{(i)}. Also, if $\alpha<\infty$, then $F$ is a Mac Lane- Vaqui\'e key polynomial for $\nu$. In particular, $\inv_\nu(F)$ is not a unit in $\mathcal G_\nu$.

Suppose now that there exists a polynomial $Q$ such that $\{\nu_i(Q)\}_{i\in I}$ is increasing and $\deg(Q)<\deg(F)$. We can take $Q$ to be monic and of smallest degree among polynomials $Q'$ with $\{\nu_i(Q')\}_{i\in I}$ increasing. Let $F=f_0+f_1Q+\ldots+f_rQ^r$ be the $Q$-expansion of $F$. For every $l$, $0\leq l\leq r$, since $\deg(f_l)<\deg(Q)$, we have that $f_l$ is $\mathfrak v$-stable. Hence, there exists $l$, $0\leq l\leq r$, such that $\nu_i(f_lQ^l)<\nu_i(f_kQ^k)$ for $k\neq l$ and $i$ large enough. Then,
\[
\nu_i(F)=\nu_i(f_lQ^l)=\nu(f_l)+l\nu_i(Q)\mbox{ for }i\mbox{ large enough}.
\]
Since $\nu_i(F)$ is increasing we must have $l>0$ and $\nu(f_0)>\nu_i(F)$ for every $i\in I$. Then
\[
\beta:=\dfrac{\nu(f_0)-\nu(f_l)}{l}>\nu_i(Q)\mbox{ for every }i\in I.
\]
Set $\nu=[\{\nu_i\}_{i\in I};\nu(Q)=\beta]$. We have that $\nu$ is a valuation and
\begin{equation}\label{equaquefazfuncipnal}
\nu(f_0)=\nu(f_lQ^l). 
\end{equation}
Then
\[
\nu(F)=\min_{1\leq j\leq r}\left\{\nu(f_j)+j\beta\right\}>\min_{1\leq j\leq r}\{\nu(f_j)+j\nu_i(Q)\}\geq \nu(f_l)+l\nu_i(Q)=\nu_i(F)
\]
for every $i\in I$. Moreover, it follows from \eqref{equaquefazfuncipnal} and Proposition 3.5 of \cite{Nart1} that $\inv_\nu(F)\notin \mathcal U_\nu$.
\end{proof}

\begin{proof}[Proof of Proposition \ref{proposparrame}] 

Take $Q\in{\rm KP}(\nu)$ such that $\phi=u\cdot \inv_\nu(Q)$ for some $u\in \mathcal U_\nu$.  If $\inv_\nu(F)=\phi^s\inv_\nu(G)$ ($\phi\nmid \inv_\nu(G)$), then $\nu(F-Q^s G)>\nu(Q^s G)=\nu(F)$. Take $\gamma\in \Gamma$ such that
\[
\nu(Q^sG)=\nu(F)<\gamma<\nu(F-Q^s G).
\]
Then we have that
\[
\alpha:=\frac{\gamma-\nu(G)}{s}>\nu(Q).
\]
Consider the valuation $\nu'=[\nu;\nu'(Q)=\alpha]$. Then
\begin{equation}\label{equatibolacha}
\nu'(F-Q^s G)\geq \nu(F-Q^s G)>\gamma=\nu'(Q^s G).
\end{equation}
Hence, $\nu'(F)=\nu'(Q^s G)=\gamma>\nu(F)$.

It follows from \eqref{equatibolacha} that $\inv_{\nu'}(Q)\mid\inv_{\nu'}(F)$ in $\mathcal G_{\nu'}$. This ends the proof.

\end{proof}
\begin{proof}[Proof of Theorem \ref{Propochata}]
Assume that $\nu\in\mathcal V_F$. By definition, there exists $\mu\in \mathcal V_{F_\infty}$ such that $\nu\leq\mu$. Since $\nu$ is non-maximal, we have $\nu<\mu$. Let $Q$ be a monic polynomial of smallest degree such that $\nu(Q)<\mu(Q)$. By Remark \ref{obsmaclane}, we have $Q\in {\rm KP}(\nu)$. If we set 
\[
\nu'=[\nu,\nu'(Q)=\mu(Q)],
\]
then $\nu<\nu'<\mu$. By Proposition \ref{propsobucara} we obtain that $\nu(F)<\nu'(F)$ and consequently by Theorem \ref{lemasobreexenso} we have $\inv_\nu(Q)\mid \inv_{\nu}(F)$. Since $\inv_\nu(Q)$ is prime we obtain that $\inv_\nu(F)\notin \mathcal U_\nu$.

For the converse, assume that $\inv_\nu(F)\notin \mathcal U_\nu$. Consider the set
\[
\mathcal S=\mathcal{V}_{F_\infty}\cup\{\nu'\in \mathcal V\mid \nu\leq \nu', F\notin \SU(\nu')\mbox{ and }\inv_{\nu'}(F)\notin\mathcal U_\nu\}.
\]
Define a partial order $\preceq$ in $\mathcal S$ by setting $\nu'\preceq \nu''$ if either $\nu'=\nu''$ or
\[
\nu'\leq \nu''\mbox{ and }\nu'(F)<\nu''(F).
\]
Since $\nu\in\mathcal S$ we have that $\mathcal S\neq\emptyset$. On the other hand, assume that $\{\nu_i\}_{i\in I}$ is a totally ordered subset of $\mathcal S$.  If $I$ has a maximum $i$, then $\nu_i\in\mathcal S$ is a majorant for $\{\nu_i\}_{i\in I}$. If $I$ does not have a maximum, then using Lemma \ref{lemausadoduasvezes} we obtain that $\{\nu_i\}_{i\in I}$ admits a majorant in $\mathcal S$. By Zorn's Lemma, $\mathcal S$ admits a maximal element $\mu$.

We claim that $\mu\in \mathcal{V}_{F_\infty}$. Indeed, if $\mu\notin \mathcal{V}_{F_\infty}$, then $F\notin\SU(\nu)$ and $\inv_\mu(F)\notin \mathcal U_\mu$ (by the definition of $\mathcal S$). By Proposition \ref{proposparrame}, there exists $\mu'\in \mathcal S$ such that $\mu(F)<\mu'(F)$. This contradicts the maximality of $\mu$ in $\mathcal S$. Therefore, $\mu\in \mathcal{V}_{F_\infty}$ and consequently $\nu\in \mathcal V_F$.
\end{proof}

\begin{proof}[Proof of Theorem \ref{corusadoprafrente}]
For each homogeneous irreducible factor $\phi_i$ of $\inv_\nu(F)$ take $Q_i\in {\rm KP}(\nu)$ such that of $\phi_i=u_i\cdot\inv_\nu(Q_i)$ for some $u_i\in\mathcal U_\nu$. Let $\nu_i\in \mathcal V$ be the valuation given by Proposition \ref{proposparrame}. By Theorem \ref{Propochata}, $\nu_i\in \mathcal V_F$ for every $i$, $1\leq i\leq r$. Since $\nu_i\in \mathcal V_F$, there exists $\mu_i\in \mathcal V_{F_\infty}$ such that $\nu_i\leq\mu_i$.

For $i,j$, $1\leq i<j\leq n$, since $\inv_\nu(Q_i)=u_i^{-1}\cdot\phi_i\nmid u_j^{-1}\cdot\phi_j=\inv_\nu(Q_j)$, by Proposition \ref{propsobucara} and Theorem \ref{lemasobreexenso}, we have that
\[
\mu_i(Q_j)=\nu(Q_j)<\mu_j(Q_j).
\]
Therefore, these valuations are distinct.

\end{proof}

\begin{figure}[!htb]
\centering
\subfigure[$\Delta_{Q,\nu}(F)$]{
\begin{tikzpicture}[line cap=round,line join=round,>=triangle 45,x=0.8cm,y=0.8cm]
      \draw[->] (-0.2,0) -- (3.2,0) node[right] {$\Q$};
    \draw[->] (0,-0.7) -- (0,2.5) node[left] {$\Gamma_\Q$};

  \filldraw[fill=blue] (1,0.5) circle (2pt);
  \filldraw[fill=blue] (2,0) circle (2pt);
  \filldraw[fill=blue] (3,0.5) circle (2pt);
  \filldraw[fill=blue] (0,1.8) circle (2pt);
  
  \draw [line width=1.pt,color=blue] (0,1.8)-- (1,0.5);
  \draw [line width=1.pt,color=blue] (1,0.5)-- (2,0);
  \draw [line width=1.pt,color=blue] (2,0)-- (3,0.5);

  \draw [dashed] (0,1)-- (1,0.5);
   \draw [dashed] (0,2)-- (3,0.5);
   
   \draw [line width=1.pt,color=black] (-0.1,1)-- (0.1,1);
\end{tikzpicture}
}
\subfigure[$\Delta_{Q,\mu}(F)$]{
\begin{tikzpicture}[line cap=round,line join=round,>=triangle 45,x=0.8cm,y=0.8cm]
      \draw[->] (-0.2,0) -- (3.2,0) node[right] {$\Q$};
    \draw[->] (0,-0.7) -- (0,2.5) node[left] {$\Gamma_\Q$};

  \filldraw[fill=blue] (1,0.5) circle (2pt);
  \filldraw[fill=blue] (2,0) circle (2pt);
  \filldraw[fill=blue] (3,0.5) circle (2pt);
  \filldraw[fill=blue] (0,1.8) circle (2pt);
  
  \draw [line width=1.pt,color=blue] (0,1.8)-- (1,0.5);
  \draw [line width=1.pt,color=blue] (1,0.5)-- (2,0);
  \draw [line width=1.pt,color=blue] (2,0)-- (3,0.5);

  \draw [dashed] (0,1.8)-- (1.385,0);
  \draw [dashed] (0.46,2)-- (2,0);
  \draw [dashed] (1.84,2)-- (3,0.5);
  
\end{tikzpicture}
}
\subfigure[$\Delta_{Q,\mu'}(F)$]{
\begin{tikzpicture}[line cap=round,line join=round,>=triangle 45,x=0.8cm,y=0.8cm]
      \draw[->] (-0.2,0) -- (3.2,0) node[right] {$\Q$};
    \draw[->] (0,-0.7) -- (0,2.5) node[left] {$\Gamma_\Q$};

  \filldraw[fill=blue] (1,0.5) circle (2pt);
  \filldraw[fill=blue] (2,0) circle (2pt);
  \filldraw[fill=blue] (3,0.5) circle (2pt);
  \filldraw[fill=blue] (0,1.8) circle (2pt);
  
  \draw [line width=1.pt,color=blue] (0,1.8)-- (1,0.5);
  \draw [line width=1.pt,color=blue] (1,0.5)-- (2,0);
  \draw [line width=1.pt,color=blue] (2,0)-- (3,0.5);

  \draw [dashed] (0,1.8)-- (1,0);
  \draw [dashed] (0,2.3)-- (1,0.5);
  \draw [dashed] (2,2.3)-- (3,0.5);
  \draw [dashed] (0.72,2.3)-- (2,0);
\end{tikzpicture}
}
\caption{Main step in the proof of Theorem \ref{mainthem}}\label{PolNewtonthem2}
\end{figure}
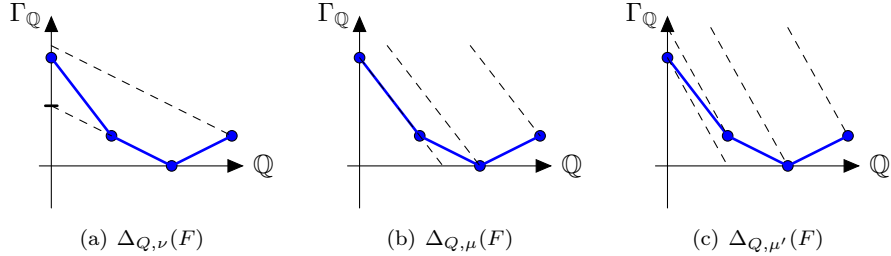

\begin{proof}[Proof of Theorem \ref{mainthem}]
Let $\phi$ be an irreducible factor of $\inv_\nu(F)$ and take $Q\in {\rm KP}(\nu)$ such that $\phi=u\cdot \inv_\nu(Q)$ for some $u\in \mathcal U_\nu$. By the proof of Proposition \ref{proposparrame}, there exists $\alpha'\in \Gamma$, $\alpha'>\nu(Q)$, such that for $\nu':=[\nu;\nu'(Q)=\alpha']$ we have
\[
\nu(Q)<\nu'(Q), \ \nu(F)<\nu'(F)\mbox{ and }\inv_{\nu'}(Q)\mid \inv_{\nu'}(F).
\]
For each $\alpha>\nu(Q)$, set $\nu_{\alpha}=[\nu;\nu_{\alpha}(Q)=\alpha]$. By definition, we have
\[
\Delta_{Q,\nu_\alpha}(F)=\Delta_{Q,\nu}(F)=\Delta_{Q,\nu'}(F).
\]
Let $\alpha_1$ be the slope of the first face of the Newton polygon $\Delta_{Q,\nu}(F)$. For each $\alpha\in\Gamma$, $\nu(Q)<\alpha\leq -\alpha_1$, by Proposition \ref{propfrelagradeainvandponl} we have $\inv_{\nu_\alpha}(F)\notin \mathcal U_{\nu_\alpha}$. Hence, we can apply Theorem \ref{Propochata} to obtain that $\nu_{\alpha}\in \mathcal V_F$. 

Set $\mu:=[\nu;\mu(Q)=-\alpha_1]$. By the above discussion, $\mu\in\mathcal V_F$. Take $\mu'\in \mathcal V$ such that $\mu\leq \mu'$ and $\mu(Q)<\mu'(Q)$. We will show that $\mu'\notin \mathcal V_F$. Indeed, since $\nu<\mu< \mu'$ we can apply Proposition \ref{propsobucara} to obtain that $\mu(a)=\mu'(a)$ for every $a\in K[x]$ with $\deg(a)<\deg(Q)$. Hence, the Newton polygons $\Delta_{Q,\mu}(F)$ and $\Delta_{Q,\mu'}(F)$ are the same. Since $\mu'(Q)>-\alpha_1$ we apply Proposition \ref{propfrelagradeainvandponl} to obtain that $\inv_{\mu'}(F)\in \mathcal U_{\mu'}$. By Theorem \ref{Propochata} we conclude that $\mu'\notin \mathcal V_F$.
\end{proof}

\section{Application to Artin-Schreier extensions}\label{ASECaso}
We start this section by presenting a result that will be essential for this section. Take $\nu\in \mathcal V$ and set $\nu_0=\nu\mid_{K}$. For every non-constant polynomial $q\in K[x]$ we consider the map
\[
\nu_{q}(f)=\min_{0\leq i\leq r}\{\nu(f_iq^i)\}
\]
where $f=f_0+\ldots+f_rq^r$ is the $q$-expansion of $f$. The map $\nu_q$ is called the \textbf{$q$-truncation of $\nu$}.

Let $p$ the characteristic exponent of $(K,\nu_0)$ (i.e., $p=1$ if ${\rm char}(K\nu_0)=0$ and $p={\rm char}(K\nu_0)$ if ${\rm char}(K\nu_0)>0$).
Set
\[
\Psi_1:=\{\nu(x-a)\mid a\in K\}
\]
and
\[
\mathcal K_1:=\{f\in K[x]\mid \nu_{x-a}(f)<\nu(f)\mbox{ for every }a\in K\}.
\]
\begin{Prop}\label{Cordahora}
Assume that $\nu$ is a rank one valuation, $\Psi_1$ is bounded in $\Gamma$ and does not have a largest element. Then, for every $G\in K[x]$ with $\deg(G)<p$ there exists $b\in K$ such that
\[
\nu(G)=\nu_{x-b'}(G)\mbox{ for every }b'\in K\mbox{ with }\nu(x-b')>\nu(x-b).
\]
\end{Prop}
\begin{proof}
By Theorem 1.1 of \cite{michael} applied to $n=1$ we conclude that the least degree of a polynomial $f\in K[x]$ for which $f\in \mathcal K_1$ is a positive power of $p$. Hence, the result follows.
\end{proof}

From now on, assume that $(K,\nu_0)$ is a valued field and ${\rm char}(K)=p>0$. Let $F\in K[x]$ be an irreducible Artin-Schreier polynomial, i.e., $F(x)=x^p-x-a$ with $a\in K$. We want to study all possible extensions of $\nu_0$ to $L:=K[x]/(F)$. We will identify these valuations with valuations $\nu$ on $K[x]$ for which $\SU(\nu)=(F)$. In this case, for $g\in K[x]$ with $\deg(g)<p$, if $\nu(g)\geq 0$, we will refer to the residue of $g$ as the residue of $g+(F)$ in $L\nu$.

We observe that if $\nu_1,\ldots,\nu_g$ are all the extensions of $\nu_0$ to $L$, then the fundamental inequality (Theorem 3.3.4 of \cite{engler}) gives us
\begin{equation}\label{fundmeineqi}
p=[L:K]\geq \sum_{i=1}^ge_if_i
\end{equation}
where $e_i$ is the ramification index and $f_i$ the inertia degree of the extension $(L|K,\nu_i)$. In particular, the maximum number of extensions of $\nu_0$ to $L$ is $p$. If there is a unique extension $\nu$ of $\nu_0$ to $L$, then we define the defect of $(L|K,\nu)$ as
\[
d=\frac{[L:K]}{ef},
\]
where $e$ and $f$ are the ramification index and inertia degree of $(L|K,\nu)$, respectively. Then
\[
p=[L:K]=def.
\]

\begin{Prop}\label{Corolsobreasc}
Let $(K,\nu_0)$ be a valued field with ${ \rm char}(K)=p>0$. Let $F(x)=x^p-x-a\in K[x]$ be an irreducible Artin-Schreier polynomial and $L=K[x]/(F)$. If there exists $b\in K$ such that $\nu_0(F(b))>0$, then there exist $p$ many distinct extensions of $\nu_0$ to $L$.
\end{Prop}
\begin{proof}
Take a valuation $\nu$ extending $\nu_0$ to $L$ and $b\in K$ such that $\nu_0(F(b))>0$. If we set $y=x-b$, then
\[
F=(y+b)^p-(y+b)-a=y^p-y+F(b).
\]
Since $\nu(F)=\infty$ and $\nu_0(F(b))>0$ we must have $\nu(y)\geq 0$. Moreover, we can suppose that $\nu(y)=0$. Indeed, if $\nu(y)>0$, then $\nu(y-1)=0$ and since $F(b)=F(b+1)$, we can replace $b$ by $b+1$ and $y$ by $y-1$.

Since $\nu(y)=0$ and $\nu_0(F(b))>0$ we have
\[
\inv_y(F)=\inv_y(y)^p-\inv_y(y)=\prod_{i=0}^{p-1}(\inv_y(y)-i)=\prod_{i=0}^{p-1}\inv_y(y-i).
\]
By Corollary \ref{corusadoprafrente}, there are (at least) $p$ many distinct valuations.
\end{proof}

\begin{figure}[!htb]
\centering
\subfigure[$\nu_0(F(b))>0$]{\label{v(a)>0}
\begin{tikzpicture}[line cap=round,line join=round,>=triangle 45,x=0.45cm,y=0.45cm]

\clip(-2.15,-1.5) rectangle (6.2,6.267584262100189);

\draw[->,color=black] (-1,0.) -- (6.15,0.);
\foreach \x in {.}
\draw[shift={(\x,0)},color=black] (0pt,2pt) -- (0pt,-2pt) node[below] {\footnotesize $\x$};

\draw[->,color=black] (0.,-1.5) -- (0.,3.1);
\foreach \y in {.}
\draw[shift={(0,\y)},color=black] (2pt,0pt) -- (-2pt,0pt) node[left] {\footnotesize $\y$};

\draw [line width=1.pt,color=blue] (0.,2.)-- (1.,0.);
\draw [line width=1.pt,color=blue] (1.,0.)-- (5.,0.);

\begin{scriptsize}
\draw [fill=blue] (0.,2.) circle (2pt);
\draw [fill=blue] (1.,0.) circle (2pt);
\draw [fill=blue] (5.,0.) circle (2pt);
\draw[color=black] (5.43,-0.55) node {$\Q$};
\draw[color=black] (-1.15,3.4) node {$\Gamma_\Q$};
\end{scriptsize}
\end{tikzpicture}
}
\subfigure[$\nu_0(F(b))=0$]{\label{v(a)=0}
\begin{tikzpicture}[line cap=round,line join=round,>=triangle 45,x=0.45cm,y=0.45cm]

\clip(-2.15,-1.5) rectangle (6.2,6.267584262100189);

\draw[->,color=black] (-1,0.) -- (6.15,0.);
\foreach \x in {.}
\draw[shift={(\x,0)},color=black] (0pt,2pt) -- (0pt,-2pt) node[below] {\footnotesize $\x$};

\draw[->,color=black] (0.,-1.5) -- (0.,3.1);
\foreach \y in {.}
\draw[shift={(0,\y)},color=black] (2pt,0pt) -- (-2pt,0pt) node[left] {\footnotesize $\y$};

\draw [line width=1.pt,color=blue] (0.,0.)-- (5.,0.);

\begin{scriptsize}
\draw [fill=blue] (0.,0.) circle (2pt);
\draw [fill=blue] (1.,0.) circle (2pt);
\draw [fill=blue] (5.,0.) circle (2pt);
\draw[color=black] (5.43,-0.55) node {$\Q$};
\draw[color=black] (-1.15,3.4) node {$\Gamma_\Q$};
\end{scriptsize}
\end{tikzpicture}
}
\subfigure[$\nu_0(F(b))<0$]{\label{v(a)<0}
\begin{tikzpicture}[line cap=round,line join=round,>=triangle 45,x=0.45cm,y=0.45cm]

\clip(-2.15,-1.5) rectangle (6.2,6.267584262100189);

\draw[->,color=black] (-1,0.) -- (6.15,0.);
\foreach \x in {.}
\draw[shift={(\x,0)},color=black] (0pt,2pt) -- (0pt,-2pt) node[below] {\footnotesize $\x$};

\draw[->,color=black] (0.,-1.5) -- (0.,3.1);
\foreach \y in {.}
\draw[shift={(0,\y)},color=black] (1pt,0pt) -- (-2pt,0pt) node[left] {\footnotesize $\y$};

\draw [line width=1.pt,color=blue] (0.,-1.)-- (5.,0.);
\begin{scriptsize}
\draw [fill=blue] (0.,-1) circle (2pt);
\draw [fill=blue] (1.,0.) circle (2pt);
\draw [fill=blue] (5.,0.) circle (2pt);
\draw[color=black] (5.43,-0.55) node {$\Q$};
\draw[color=black] (-1.15,3.4) node {$\Gamma_\Q$};
\end{scriptsize}
\end{tikzpicture}
}
\caption{Possible Newton polygons $\Delta_{x-b}(F)$ for $b\in K$}\label{PolNewtonFxpxa}
\end{figure}
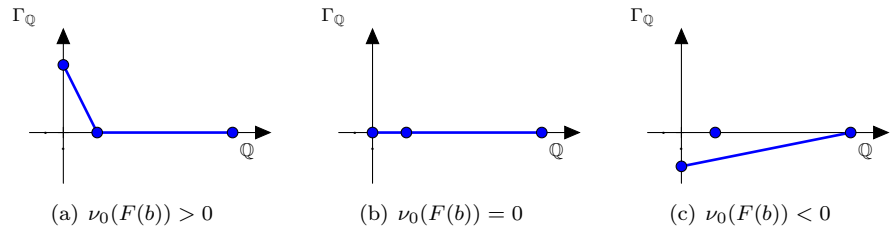

\begin{Obs}\label{obskuhlmanecs}
Suppose that $\nu_0(F(b))=0$ for some $b\in K$. Then for $c\in K$, we have $\nu_0(F(c))>0$ if and only if $(c-b)\nu_0$ is a root of $\overline F:=X^p-X+F(b)\nu_0\in K\nu_0[X]$. Since $\overline F$ is an Artin-Schreier polynomial, either it has a root in $K\nu_0$ or it is irreducible.
\end{Obs}

We now present a proof of Theorem \ref{kuhlmannrequs}
\begin{proof}[Proof of Theorem \ref{kuhlmannrequs}]
If there exists $b\in K$ such that $\nu_0(F(b))>0$, then by Proposition \ref{Corolsobreasc} there exist $p$ many distinct extensions of $\nu_0$ to $L$. The converse will follow from \textbf{(ii)} and \eqref{fundmeineqi}.

In order to prove \textbf{(ii)} assume that $\nu_0(F(b))\leq 0$ for every $b\in K$. Hence, $\Delta_{x-b}(F)$ has only one side with slope $-\frac{\nu_0(F(b))}{p}$. For any extension $\mu$ of $\nu_0$ to $L$, by Proposition \ref{Popdiscumath}, we have
\begin{equation}\label{equatnutulabaaixo}
\mu(x-b)=\frac{\nu_0(F(b))}{p}.
\end{equation}
In particular, if $\frac{\nu_0(F(b))}{p}\notin \Gamma_0$ for some $b\in K$, then $e=p$.

Assume that $S\subseteq \Gamma_0$ and that $S$ has a maximum. If there exists $b\in K$ such that $\nu_0(F(b))=0$, then by Remark \ref{obskuhlmanecs},
\[
X^p-X+F(b)\nu_0\in K\nu_0[X]
\]
is irreducible (and the residue of $x-b$ is a root of it). Hence $f=p$.

Take $b\in K$ and suppose that $\nu_0(F(b))<0$ and that $\displaystyle\frac{\nu_0(F(b))}{p}$ is a maximum of $S$. Take $c\in K$ such that $\nu_0(c)=\dfrac{\nu_0(F(b))}{p}$ (such $c$ exists because $S\subseteq \Gamma_0$). Then
\begin{displaymath}
\begin{array}{rcl}
\displaystyle\nu\left(\left(\frac{x-b}{c}\right)^p+\frac{F(b)}{c^p}\right)&=&\displaystyle\nu\left(\frac{x^p-b^p +b^p-b-a}{c^p}\right)=\nu\left(\frac{F(x)+x-b}{c^p}\right)\\[8pt]
&=&\nu(x-b)-p\nu_0(c)=(1-p)\nu_0(c)>0.
\end{array}
\end{displaymath}

In particular, if we denote by $\eta$ the residue of $\displaystyle\frac{x-b}{c}$, then $\eta$ is a root of
\[
\overline F_b=X^p+\frac{F(b)}{c^p}\nu_0\in K\nu_0[X].
\]
Since ${\rm char}(K\nu_0)=p$, we have that either $\eta\in K\nu_0$ or $\overline F_b$ is irreducible. By \eqref{fundmeineqi}, we conclude that if $\overline F_b$ is irreducible, then $f=p$ (and $g=e=1$). 
We will show that $\eta\notin K\nu_0$. Otherwise, there would exist $c'\in \VR_{\nu_0}$ such that
\[
0<\nu\left(\frac{x-b}{c}-c'\right)=\nu\left(\frac{x-b-cc'}{c}\right).
\]
In this case, by \eqref{equatnutulabaaixo} we have
\[
\dfrac{\nu_0(F(b+cc'))}{p}=\nu\left(x-b-cc'\right)>\nu_0(c)=\dfrac{\nu_0(F(b))}{p}.
\]
This is a contradiction to the maximality of $\dfrac{\nu_0(F(b))}{p}$ in $S$. Therefore, $f=p$.

Assume that $S\subseteq \Gamma_0$ and that $S$ does not have a maximum. Take any valuation $\mu$ (possibly different from $\nu$) extending $\nu_0$. By \eqref{equatnutulabaaixo} we have
\[
\nu(x-b)=\dfrac{\nu_0(F(b))}{p}=\mu(x-b)
\]
and since $\nu|_K=\nu_0=\mu|_K$ we obtain $\nu_{x-b}=\mu_{x-b}$. Moreover, for every polynomial $G(x)$ of degree smaller than $p$, by Proposition \ref{Cordahora} (here is the only instance where we use that the rank of $\nu_0$ is one), there exists $b\in K$ such that if $\nu(x-b')>\nu(x-b)$, then
\begin{equation}\label{eqbafquehju}
\nu(G)=\nu_{x-b'}(G)=\mu_{x-b'}(G)=\mu(G).
\end{equation}
Hence, $\nu$ is the unique valuation on $L$ extending $\nu_0$. It follows from \eqref{eqbafquehju}, the definition of $\nu_{x-b'}$ and the fact that $S\subseteq \Gamma_0$ that $e=1$.

It remains to show that $f=1$. Take any polynomial $G\in K[x]$ with $\deg(G)<p$ and take $b\in K$ such that
\[
\nu_{x-b}(G)=\nu(G).
\]
This implies that $\inv_{x-b}(x-b)\nmid \inv_{x-b}(G)$. Since $S$ is not bounded, there exists $b'\in K$ such that $\nu(x-b')>\nu(x-b)$. In particular, $\inv_{x-b}(x-b)\mid \inv_{x-b}(x-b')$. Altogether, we conclude that
\begin{equation}\label{equnrefeersobreinvs}
\inv_{x-b}(x-b')\nmid \inv_{x-b}(G).
\end{equation}
Let
\[
G=g_0+g_1(x-b')+\ldots+g_s(x-b')^s
\]
be the $x-b'$-expansion of $G$. By \eqref{equnrefeersobreinvs}, we obtain
\[
\nu(G)=\nu(g_0)<\nu\left(g_i(x-b')^i\right),\mbox{ for every }i, 1\leq i\leq s.
\]
This means that if $\nu(G)=0$, then the residues of $G$ and $g_0$ are the same. Therefore, $f=1$.
\end{proof}
\begin{Obs}
Proposition \ref{lemausadoduasvezes} and Theorem \ref{kuhlmannrequs} are also true for valuations with arbitrary rank. However, for their proofs, we use Theorem 1.1 of \cite{michael} whose statement is only true for rank  one valuations. Since the goal of this section is to illustrate the theory developed here, we will not state them in their whole generality. 
\end{Obs}
\begin{Obs}
The crucial step in the above proof used Proposition \ref{Cordahora}. One could use, alternatively, the theory of \textit{pseudo-convergent sequences} and Lemma 8 of \cite{Kapl} to deduce the same result.
\end{Obs}

\noindent{\footnotesize MATHEUS DOS SANTOS BARNAB\'E\\
Departamento de Matem\'atica--UFSCar\\
Rodovia Washington Lu\'is, 235\\
13565-905, S\~ao Carlos - SP, Brazil.\\
Email: {\tt matheusbarnabe@dm.ufscar.br} \\\\

\noindent{\footnotesize JOSNEI NOVACOSKI\\
Departamento de Matem\'atica--UFSCar\\
Rodovia Washington Lu\'is, 235\\
13565-905, S\~ao Carlos - SP, Brazil.\\
Email: {\tt josnei@ufscar.br}
\end{document}